\documentclass[letterpage, 11pt, notitlepage]{article}

\usepackage[margin=1.0in]{geometry}

\usepackage{amsmath, amsfonts, amssymb, graphicx,epstopdf,color,soul,mathtools,enumitem}
\usepackage{mathrsfs, mathtools,bbm, dsfont}

\usepackage{ifthen}
\usepackage{multirow, multicol}
\usepackage{float}
\usepackage{subeqnarray, array}
\usepackage{enumitem}
\usepackage{xcolor}
\usepackage[font=small]{caption}
\usepackage{tikz}
\usepackage{times}
\usepackage{amsthm}

\usepackage[bookmarks=true,pageanchor,colorlinks,linkcolor=red,anchorcolor=blue, citecolor=blue,urlcolor=blue,hyperfootnotes=false]{hyperref}

\newcommand{\Eset}{\mathbb{E}}

\newcommand{\Rset}{\mathbb{R}}

\newcommand{\Ocal}{{\cal O}}

\newcommand{\Qcal}{{\cal Q}}

\newcommand{\zhat}{{\hat{z}}}

\newcommand{\xhat}{{\hat{x}}}
\newcommand{\yhat}{{\hat{y}}}

\newtheorem{lem}{Lemma}
\newtheorem{thm}{Theorem}

\newtheorem{assump}{Assumption}
\newtheorem{remark}{Remark}

\date{}

\title{\LARGE Nonlinear Two-Time-Scale Stochastic Approximation\\ \LARGE Convergence and Finite-Time Performance}
\author{
Thinh T. Doan\thanks{Thinh T. Doan is with the Bradley Department of Electrical and Computer Engineering, Virginia Tech, USA. Email: {\tt\small thinhdoan@vt.edu}}
}

\begin{document}

\maketitle

\begin{abstract}
Two-time-scale stochastic approximation, a generalized version of the popular stochastic approximation, has found broad applications in many areas including stochastic control, optimization, and machine learning. Despite its popularity, theoretical guarantees of this method, especially its finite-time performance, are mostly achieved for the linear case while the results for the nonlinear counterpart are very sparse. Motivated by the classic control theory for singularly perturbed systems, we study in this paper the asymptotic convergence and finite-time analysis of the nonlinear two-time-scale stochastic approximation. Under some fairly standard assumptions, we provide a formula that explicitly characterizes the rate of convergence of the main iterates to the desired solutions. In particular, we show that the mean square error generated by the method convergences to zero at a rate $\Ocal(1/k^{2/3})$, where $k$ is the number of iterations. The key idea in our analysis is to properly choose the two step sizes to  characterize the coupling between the fast and slow-time-scale iterates.             
\end{abstract}


\section{Nonlinear two-time-scale stochastic approximation}\label{sec:nonlinear_SA}
Stochastic approximation ({SA}), introduced by \cite{RobbinsM1951}, is a simulation-based approach for finding the root (or fixed point) of some unknown operator $F$ represented by the form of an expectation, i.e., $F(x) = \Eset_{\pi}[F(x,\xi)]$, where $\xi$ is some random variable with a distribution $\pi$. Specifically, this method seeks a point $x^{\star}$ such that $F(x^{\star}) = 0$ based on the noisy observations $F(x;\xi)$. The iterate $x$ is iteratively updated by moving along the direction of $F(x;\xi)$ scaled by some step size. Through a careful choice of this step size, the ``noise'' induced by the random samples $\xi$ can be averaged out across iterations, and the algorithm converges to $x^*$.  {SA} has found broad applications in many areas including statistics, stochastic optimization, machine learning, and reinforcement learning \cite{BTbook1999,borkar2008,TTFBook2009,SBbook2018,LanBook2020}.      

In this paper, we consider the so-called two-time-scale SA, a generalized variant of the classic SA, which is used to find the root of a system of two coupled nonlinear equations. Given two unknown operators $F:\Rset^{d}\times\Rset^{d}\rightarrow\Rset^{d}$ and $G:\Rset^{d}\times\Rset^{d}\rightarrow\Rset^{d}$, we seek to find $x^{\star}$ and $y^{\star}$ such that
\begin{align}
\left\{
\begin{aligned}
&F(x^{\star},y^{\star}) = 0\\
&G(x^{\star},y^{\star}) = 0.
\end{aligned}\right.\label{prob:FG}
\end{align}
Since $F$ and $G$ are unknown, we assume that there is a stochastic oracle that outputs noisy values of $F(x,y)$ and $G(x,y)$ for a given pair $(x,y)$. In particular, for any given points $x$ and $y$ we have access to $F(x,y)\, +\, \xi$ and $G(x,y)\, +\, \psi$, where $\xi$ and $\psi$ are two random variables. Using this stochastic oracle, we study the two-time-scale nonlinear SA for solving problem \eqref{prob:FG}, which iteratively updates the iterates $x_{k}$ and $y_{k}$, the estimates of $x^{\star}$ and $y^{\star}$, respectively, for any $k\geq 0$ as
\begin{align}
\begin{aligned}
x_{k+1} &= x_{k} - \alpha_{k}\left(F(x_{k},y_{k}) + \xi_{k}\right)\\   
y_{k+1} &= y_{k} - \beta_{k}\left(G(x_{k},y_{k}) + \psi_{k}\right),
\end{aligned}\label{alg:xy}
\end{align}
where $x_{0}$ and $y_{0}$ are arbitrarily initialized in $\Rset^{d}$. In \eqref{alg:xy}, $\alpha_{k}$ and $\beta_{k}$ are two nonnegative step sizes chosen such that $\beta_{k}\ll \alpha_{k}$, i.e., the second iterate is updated using step sizes that are very small as compared to the ones used to update the first iterate.  Thus, the update of $x_{k}$ is referred to as the ``fast-time scale" while the update of $y_{k}$ is called the ``slow-time scale". The time-scale difference here is loosely defined as the ratio between the two step sizes, i.e., $\beta_{k}/\alpha_{k}$. In addition, the update of the \textit{fast iterate} depends on the \textit{slow iterate} and vice versa, that is, they are coupled to each other. To handle this coupling, the two step sizes have to be properly chosen to guarantee the convergence of the method. Indeed, an important problem in this area is to select the two step sizes so that the two iterates converge as fast as possible. Our main focus is, therefore, to derive the convergence rate of \eqref{alg:xy} in solving \eqref{prob:FG} under some proper choice of these two step sizes and to understand their impact on the performance of the nonlinear two-time-scale SA.

\subsection{Main contributions}
The focus of this paper is to derive the asymptotic convergence and finite-time performance of the nonlinear two-time-scale SA. In particular, we provide a formula that characterizes the rate of convergence of the main iterates to the desired solutions. Under some proper choice of step sizes $\alpha_{k}$ and $\beta_{k}$, we show that the mean square error generated by the method converges at a rate $\Ocal(1/k^{2/3})$, where $k$ is the number of iterations. Our key technique is motivated by the classic control theory for singularly perturbed systems, that is, we  properly choose the two step sizes to characterize the coupling between the fast and slow-time-scale iterates. In addition, our convergence rate analysis also provides an insight for different choice of step sizes under different settings.  

\subsection{Motivating applications}\label{sec:motivating_applications}
Nonlinear two-time-scale SA, Eq.\ \eqref{alg:xy}, has been found in numerous applications within engineering and sciences. One concrete example is to model the well-known stochastic gradient descent (SGD) with the Polyak-Rupert averaging for minimizing an objective function $f$, i.e., we want to solve 
\begin{align*}
\underset{y\in\Rset^{d}}{\text{minimize }} f(y)  
\end{align*}
where we only have access to the noisy observations of the true gradients. In this case, the classic SGD iteratively updates the iterates $y_{k}$ as \begin{align*}
y_{k+1} = y_{k} - \beta_{k}(\nabla f(y_{k}) + \psi_{k}),    
\end{align*}
where $\psi_{k}$ is some zero-mean random variables. In order to improve the convergence of SGD, an additional averaging step is often used \cite{PolyakJ1992,Ruppert88}
\begin{align*}
x_{k+1} = \frac{1}{k+1}\sum_{t=0}^{k}y_{k} &= x_{k} + \frac{1}{k+1}\left(y_{k} - x_{k}\right).      
\end{align*}
Obviously, these two updates are special case of the nonlinear two-time-scale SA in \eqref{alg:xy}. A more complicated version of two-time-scale SGD for optimizing the composite function $f(g(x))$ can also be viewed as a variant of \eqref{alg:xy} \cite{WangFL2017, ZhangX2019}. In addition, two-time-scale methods have been used in distributed optimization to address the issues of communication constraints 
\cite{DoanBS2017,DoanMR2018b} and in distributed control to handle clustered networks \cite{Romeres13,JChow85,Biyik08,Boker16,ThiemDN2020}. Finally, two-time-scale SA has been used extensively to model  reinforcement learning methods, for example, gradient temporal difference (TD) learning  and actor-critic methods \cite{Sutton2009a,Sutton2009b,Maeietal2009,KondaT2003,xu_actor_critic2020,wu_actor_critic2020,Hong_actor_critic2020,Khodadadian_actorcritic_2021}. As a specific application of \eqref{alg:xy} in reinforcement learning, we consider the gradient TD learning for solving the policy evaluation problem under nonlinear function approximations studied in \cite{Maeietal2009}, which can be viewed as a variant of \eqref{alg:xy}. In this problem, we want to estimate the cumulative rewards $V$ of a stationary policy using function approximations $V_{y}$, that is, our goal is to find $y$ so that $V_{y}$ is as close as possible to the true value $V$. Here, $V_{y}$ can be represented by a neural network where $y$ is the weight vector of the network. Let $\zeta$ be the environmental sate, $\gamma$ be the discount factor, $\phi(\zeta) = \nabla V_{y}(\zeta)$ be the feature vector of state $\zeta$, and $r$ be the reward return by the environment. Given a sequence of samples $\{\zeta_{k},r_{k}\}$, one version of  {GTD} are given as
\begin{align*}
x_{k+1} &= x_{k} + \alpha_{k}( \delta_{k} - \phi(\zeta_{k})^{T}x_{k})\phi(\zeta_{k})\\
y_{k+1} &= y_{k} + \beta_{k}\Big[\left(\phi(\zeta_{k})-\gamma\phi(\zeta_{k+1})\right)\phi(\zeta_{k})^Tx_{k} - h _{k}\Big],
\end{align*}
where $\delta_{k}$ and $h_{k}$ are defined as  
\begin{align*}
\delta_{k} &= r_{k} + \gamma V_{y_{k}}(\zeta_{k+1}) - V_{y_{k}}(\zeta_{k})\\ 
h_{k} &= (\delta_{k} - \phi(\zeta_{k})^Tx_{k})\nabla^2V_{y_{k}}(\zeta_{k})x_{k},    
\end{align*}
which is clearly a variant of \eqref{alg:xy} under some proper choice of $F$ and $G$. It has been observed that the {GTD} method is more stable and performs better compared to the single-time-scale counterpart (TD learning) under off-policy learning and nonlinear function approximations. 


\subsection{Related works}
Given the broad applications of SA in many areas, its convergence properties have received much interests for years. In particular, the asymptotic convergence of SA, including its two-time-scale variant, can be established by using the (almost) Martingale convergence theorem  when the noise are i.i.d or the ordinary differential equation (ODE) method for more general noise settings; see for example \cite{borkar2008,benveniste2012adaptive,BTbook1999}. Under the right conditions both of these methods show that the noise effects eventually average out and the SA iterate asymptotically converges to the desired solutions. 

The convergence rate of the single-time-scale SA has been mostly studied in the context of SGD under the i.i.d noise model; see for example \cite{BottouCN2018} and the references therein. Given the wide applications of SA in reinforcement learning, three are  significant interests in analyzing the finite-time analysis of SA under different conditions; see for example \cite{Bhandari2018_FiniteTD,Karimi_colt2019, SrikantY2019_FiniteTD,HuS2019,Chen_MC_LinearSA_2020} for linear SA and \cite{ChenZDMC2019} for nonlinear counterpart.  

Unlike the single-time-scale SA, the convergence rates of the two-time-scale SA are less understood due to the complicated interactions between the two step sizes and the iterates.  Specifically, the rates of the two-time-scale SA has been studied mostly for the linear settings, i.e, when $F$ and $G$ are linear functions w.r.t their variables; see for example in \cite{KondaT2004, DalalTSM2018, DoanR2019,GuptaSY2019_twoscale,Doan_two_time_SA2019,Kaledin_two_time_SA2020,Dalal_Szorenyi_Thoppe_2020}. On the other hand, we are only aware of the work in \cite{MokkademP2006}, which considers the finite-time analysis of the nonlinear two-time-scale SA in \eqref{alg:xy}. In particular, under the stability condition (Assumption $1$ in \cite{MokkademP2006}, $lim_{k\rightarrow\infty}(x_{k},y_{k}) = (x^{\star},y^{\star}))$ and when $F$ and $G$ can be locally approximated by linear functions in a neighborhood of $(x^{\star},y^{\star})$, a convergence rate of \eqref{alg:xy} in distribution is provided. They also show that the rates of the fast-time and slow-time scales are asymptotically decoupled under proper choice of step sizes, which agrees with the previous observations of two-time-scale SA; see for example \cite{KondaT2004}. In this paper, our focus is to study the finite-time analysis that characterizes the rates of \eqref{alg:xy} in mean square errors. We do this under different assumptions on the operators $F$ and $G$ as compared to the ones considered in \cite{MokkademP2006}; for example, we do not require the stability condition. Our setting is motivated by the conditions considered in \cite[Chapter 7]{Kokotovic_SP1999}, where the authors study the continuous-time and deterministic version of \eqref{alg:xy}, i.e., $\xi_{k} = \psi_{k} = 0$.     


\section{Main Results}\label{sec:results}
In this section, we present in details the main results of this paper, that is, we provide a finite-time analysis for the convergence rates of the two-time-scale stochastic approximation in \eqref{alg:xy}. Under some certain conditions explained below, we show that the mean square errors of these iterates converge to zero at a rate
\begin{align*}
&\Eset\left[\|y_{k}-y^{\star}\|^2\right] + \frac{\beta_{k}}{\alpha_{k}}\Eset\left[\|x_{k}-x^{\star}\|^2\right]\notag\\ 
&\leq \Ocal\left( \frac{\Eset[\|y_{0}-y^*\|^2+\|x_{0}-x^*\|^2]}{(k+1)^2}\frac{1}{(k+1)^{2/3}}\right),    
\end{align*}
where the choice of $\beta_{k}\ll\alpha_{k}$ is discussed in the next section. In addition, under the same choice of step sizes we also obtain
\begin{align*}
 \lim_{k\rightarrow\infty}\|x_{k}-x^{\star}\| = \lim_{k\rightarrow\infty}\|y_{k}-y^{\star}\| = 0\qquad a.s.   
\end{align*}
The main idea of our analysis is based on a proper choice of the two step sizes and a Lyapunov function that takes into account the coupling between the two iterates. Before presenting the details of our results, we discuss the main idea behind our approach and assumptions, motivated by the ones used to study the stability of the corresponding different equations of \eqref{alg:xy}  \cite{SaberiK_perturbation1984,Kokotovic_SP1999} .  
\subsection{Main Ideas}
For an ease of exposition, we assume for the moment that the step sizes are constants, i.e., $\beta_{k} \equiv \beta \ll \alpha \equiv \alpha_{k}$. As mentioned, the asymptotic convergence of SA can be established by using the ODE method, where under some proper choice of step sizes the noise effects in SA eventually average out and the SA follows the solution of suitably defined differential equations. In the case of nonlinear two-time-scale SA, it is reduced to the following two differential equations
\begin{align}
\begin{aligned}
\frac{dx}{dt} &= -F(x(t),y(t))\\   
\frac{dy}{dt} &= -\frac{\beta}{\alpha}G(x(t),y(t)),
\end{aligned}\label{alg:xy_ODE}
\end{align}
where the ratio $\beta/\alpha$ represents the difference in time scale between these two updates. We assume that $(x^*,y^*)$ is the unique equilibrium \eqref{alg:xy_ODE}.  The stability of \eqref{alg:xy_ODE} studied in \cite[Chapter 7]{Kokotovic_SP1999} can be argued as follows.  Since $\beta/\alpha \ll 1$, the dynamic of $x(t)$ evolves much faster than the one of $y(t)$. Thus, one can consider $y(t) = y$ being fixed in $\dot{x}$ and separately study the stability of the following ``fast" system  
\begin{align*}
\dot{x}(t) = - F(x(t),y).    
\end{align*}
Given a fixed $y$, to study the stability of $x(t)$ we first consider the condition on the existence of a solution of  this different equation, i.e., there exists an operator $H$ such that $x = H(y)$ is the unique solution of $F(x,y) = 0$ or the unique equilibrium of $\dot{x}$. A common condition is to assume that $F$ is Lipschitz continuous. Second, to show that $x(t)$ converges to this solution under any initial condition this solution must be unique. The convergence of $x(t)$ is then studied by searching for a candidate Lyapunov function $V_{F}(x,y)$, such that the following condition holds with some constant $\mu_{x}>0$
\begin{align*}
\frac{dV_{F}(x(t),y)}{dt} &= -\frac{\partial V_{F}}{\partial x}F(x(t),y) \leq -\mu_{x}\|x(t)-H(y)\|^2\\
V_{F}(x(t),y) &= 0\; \text{ iff }\; x(t) = H(y).
\end{align*}
In this case, Lyapunov theorem says that under the right conditions of the Lyapunov functions $x(t)$ converges exponetially to $H(y)$ \cite{Khalil2002}. This exponential stability is established based on the condition that $y(t)$ is held fixed in the update of $\dot{x}$. In our two-time-scale differential equation \eqref{alg:xy_ODE}, $y(t)$ is also changing, therefore, the derivative of $V_{F}$ is given as
\begin{align}
&\frac{d V_{F}(x(t),y(t))}{dt} =    -\frac{\partial V_{F}}{\partial x}F(x(t),y(t))  - \frac{\beta}{\alpha}\frac{\partial V_{F}}{\partial y}G(x(t),y(t)),\label{ODE:Eq1}
\end{align}
which does not immediately yield the convergence of $x(t)$ to $H(y(t))$ unless some additional requirements on the asymptotic convergence of $y(t)$ is satisfied. On the other hand, we consider the slow dynamic $\dot{y}$ as follows 
\begin{align*}
\dot{y}(t) = &-\frac{\beta}{\alpha}G(H(y(t)),y(t)) - \frac{\beta}{\alpha}\Big[G(x(t),y(t)) - G(H(y(t)),y(t))\Big].    
\end{align*}
Similarly, we assume that $G$ is Lipschitz continuous so that there exists a (unique) solution of this differential equation. To study the stability of this system, we again consider another candidate Lyapunov function when $x(t) = H(y(t))$, i.e., we assume that there exists a Lyapunov function $V_{G}(y)$ and a positive constant $\mu_{y}$ satisfying
\begin{align*}
\frac{d V_{G}(y(t))}{dt} = -\frac{\beta}{\alpha}\frac{\partial V_{G}}{\partial y}G(H(y(t)),y(t)) \leq -\mu_{y}\|y(t)\|^2,      
\end{align*}
which again implies that $y(t)$ converges exponentially to the equilibrium of $\dot{y}(t)$. However, in our case since $x(t)\neq H(y(t))$ we have
\begin{align}
\frac{d V_{G}(y(t))}{dt} &= -\frac{\beta}{\alpha}\frac{\partial V_{G}}{\partial y}G(H(y(t)),y(t)) - \frac{\beta}{\alpha}\frac{\partial V_{G}}{\partial y}\Big[G(x(t),y(t)) - G(H(y(t)),y(t))\Big]   \notag\\
& \leq -\mu_{y}\|y(t)\|^2 + \frac{L_{G}\beta }{\alpha}\left\|\frac{\partial V_{G}}{\partial y}\right| \left\|x(t)- H(y(t))\right\|,\label{ODE:Eq2}  
\end{align}
where we assume that $G$ is Lipschitz continuous with constant $L_{G}>0$. 
In this case, one cannot immediately conclude the convergence of $y(t)$, unless the asymptotic convergence of $x(t)$ to $H(y(t))$ is guaranteed at a proper rate. Thus, to study the stability of $x(t)$ and $y(t)$ one needs to study their convergences simultaneously. One way to achieve this goal is to consider the candidate Lyapunov function $V$, which is an aggregate of $V_{G}$ and $V_{F}$. Using this $V$ one can characterize the coupling as well as the convergence of these variables. In this paper, we use the following Lyapunov function based on the time-coupling ratio $\beta/\alpha$ between the fast and slow updates
\begin{align*}
V(x,y) = V_{F}(x,y) + \frac{\beta}{\alpha}V_{G}(x,y).     
\end{align*}
In \cite{Kokotovic_SP1999}, the authors consider a convex combination of $V_{F}$ and $V_{G}$ to study the stability of \eqref{alg:xy_ODE}. We find that it is more convenient to use the Lyapunov function above in studying the convergence rates of the two-time-scale stochastic approximation.  

Our settings and analysis in the sequel are established based on the observation explained above. Due to the impact of the noise, one cannot in general guarantee that the solutions of the stochastic systems will track the ones of \eqref{alg:xy_ODE}. Under some fairly standard assumptions, similar to the ones used in studying the stability of \eqref{alg:xy_ODE}, we establish the asymptotic convergence of $(x_{k},y_{k})$ to $(x^*,y^*)$. In addition, we provide a finite-time bound to characterize the rates of this convergence.    


\subsection{Preliminaries and Assumptions}
In this section we present the main assumptions and some preliminaries, which are useful for our later analysis. Our first assumption is on the smoothness of $F$ and $G$, which basically guarantees the existence of the solutions of \eqref{alg:xy_ODE}.
\begin{assump}\label{assump:smooth:FH}
Given $y\in\Rset^{d}$ there exists an operator $H:\Rset^{d}\rightarrow\Rset^{d}$ such that $x = H(y)$ is the unique solution of
\begin{align*}
F(H(y),y) = 0,    
\end{align*}
where $H$ and $F$ are Lipschitz continuous with positive constant $L_{H}$ and $L_{F}$, respectively, i.e., $\forall x_{1}, x_{2}, y_{1}, y_{2} \in\Rset^{d}$
\begin{align}
&\|H(y_{1}) - H(y_{2})\| \leq L_{H}\|y_{1}-y_{2}\|,\label{assump:smooth:FH:ineqH}\\
&\|F(x_{1},y_1) - F(x_{2},y_2)\| \leq L_{F}(\|x_{1}-x_{2}\| + \|y_{1} - y_{2}\|).     \label{assump:smooth:FH:ineqF}
\end{align}
\end{assump}
\begin{remark}
In the case of the linear two-time-scale SA, e.g., $F$ and $G$ are linear operators 
\begin{align*}
F(x,y) &= A_{11}x+A_{12}y\\
G(x,y) &= A_{12}x + A_{22}y,
\end{align*}
$H$ is also a linear operator, i.e., $H = A_{11}^{-1}A_{12}y$ where $A_{11}$ is assumed to be a negative definite (not necessarily symmetric) matrix \cite{KondaT2004,Doan_two_time_SA2019}. 
\end{remark}
Second, for the global asymptotic convergence of $x(t)$ to the equilibrium of \eqref{alg:xy_ODE}  it is necessary that this equilibrium is unique. This condition is guaranteed if $F$ is strong monotone. 
\begin{assump}\label{assump:sm:F}
$F$ is strongly monotone w.r.t $x$  when $y$ is fixed, i.e., there exists a constant $\mu_{F} > 0$ 
\begin{align}
\left\langle x - z, F(x,y) - F(z,y) \right\rangle \geq \mu_{F} \|x-z\|^2. \label{assump:sm:F:ineq}    
\end{align}
\end{assump}
Similarly, we consider the same  assumptions to guarantee the existence and uniqueness of the equilibrium of $\dot{y}$ in  \eqref{alg:xy_ODE}. 
\begin{assump}\label{assump:G}
The operator $G(\cdot,\cdot)$ is Lipschitz continuous with constant $L_{G}$, i.e., $\forall x_{1}, x_{2}, y_{1}, y_{2} \in\Rset^{d}$,
\begin{align}
\hspace{-0.3cm}\|G(x_{1},y_{1}) - G(x_{2},y_{2})\| \leq L_{G}\left(\|x_{1} - x_{2}\| + \|y_{1} - y_{2}\| \right).    \label{assump:G:smooth}
\end{align}
Moreover, $G$ is $1$-point strongly monotone w.r.t $y^{\star}$, i.e., there exists a constant $\mu_{G} > 0$ such that
\begin{align}
\left\langle y - y^{\star}, G(H(y),y) \right\rangle \geq \mu_{G} \|y - y^{\star}\|^2, \quad \forall y\in\Rset^{d}. \label{assump:G:sm}    
\end{align}
\end{assump}
Note that the Lipschitz continuity of $G$ on both variables is necessary since it guarantees the existence of the solutions of $\dot{y}$ when $x(t) = H(y(t))$ presented in the previous section. These two assumptions are also considered under different variants in the context of both linear and nonlinear two-time-scale {\sf SA} studied in \cite{KondaT2004, DalalTSM2018, DoanR2019,GuptaSY2019_twoscale,Doan_two_time_SA2019,Kaledin_two_time_SA2020,MokkademP2006}. Finally, we consider i.i.d noise models, that is, $\xi_{k}$ and $\psi_{k}$ are Martingale difference. We denote by $\Qcal_{k}$ the filtration containing all the history generated by \eqref{alg:xy} up to time $k$, i.e.,
\begin{equation*}
\Qcal_{k} = \{x_{0},y_{0},\xi_{0},\psi_{0},\xi_{1},\psi_{1},\ldots,\xi_{k},\psi_{k}\}.
\end{equation*}
\begin{assump}\label{assump:noise}
The random variables $\xi_{k}$ and $\psi_{k}$, for all $k\geq0$, are independent of each other and across time, with zero mean and common variances given as follows
\begin{equation}
\Eset[\xi_{k}^T\xi_{k}\,|\,\Qcal_{k-1}] = \Gamma_{11},\quad
\Eset[\psi_{k}^T\psi_{k}\,|\,\Qcal_{k-1}] = \Gamma_{22}.
\label{assump:variance}
\end{equation}
\end{assump}
We conclude this section by presenting some preliminaries, which will be used in deriving our main results in the next subsection. The proofs of these lemmas can be found in Section \ref{sec:preliminaries}. To the end of this paper, we consider nonincreasing and nonnegative time-varying sequence of step sizes $\{\alpha_{k},\beta_{k}\}$, e.g., $\alpha_{k}\leq \alpha_{0}$, $\beta_{k}\leq \beta_{0}$, and $\beta_{k}\leq\alpha_{k}$. To study the convergence rate of $(x_{k},y_{k})$ to $(x^*,y^*)$, one can consider the  mean square errors $\Eset[\|x_{k}-x^*\|^2]$ and $\Eset[\|y_{k}-x^*\|^2]$. However, these errors do not immediately give the relationship between the fast and slow variables in \eqref{alg:xy}. On the other hand, as we discussed in the previous section the fast and slow-time-scale updates are coupled through the term $x-H(y)$. Therefore, we introduce the following two residual variables, which are more natural to characterize the coupling between $x_{k}$ and $y_{k}$ and to derive the rates of \eqref{alg:xy}
\begin{align}
\begin{aligned}
\xhat_{k} &= x_{k} - H(y_{k})\\   
\yhat_{k} &= y_{k} - y^{\star}.
\end{aligned}    \label{alg:xyhat}
\end{align}
Obviously, if $\yhat_{k}$ and $\xhat_{k}$ go to zero, $(x_{k},y_{k})\rightarrow(x^{\star},y^{\star})$. Thus, to establish the convergence of $(x_{k},y_{k})$ to $(x^{\star},y^{\star})$ one can instead study the convergence of $(\xhat_{k},\yhat_{k})$ to zero. To do that, we first consider the relation of these two residual variables in the following two lemmas. For convenience, we denote by 
\begin{align}
\gamma_{k} =  \frac{(1+L_{F}\alpha_{k})^2\beta_{k}^2}{\mu_{F}\alpha_{k}}\cdot   \label{notation:gamma}
\end{align}

\begin{lem}\label{lem:xhat}
Suppose that Assumptions \ref{assump:smooth:FH}  -- \ref{assump:noise} hold. Let $\{x_{k},y_{k}\}$ be generated by \eqref{alg:xy}. Then, we have for all $k\geq 0$
\begin{align}
\Eset\left[\|\xhat_{k+1}\|^2\,|\,\Qcal_{k}\right] &\leq \left(1-\mu_{F}\alpha_{k} \right) \|\xhat_{k}\|^2 +  \beta_{k}^2\Gamma_{22} +  \alpha_{k}^2\Gamma_{11} + 4L_{H}^2\gamma_{k}\Gamma_{22}\notag\\ 
&\quad + L_{H}^2\left(2L_{G}^2\beta_{k}^2 + \alpha_{k}^2+ 4L_{G}^2\gamma_{k}\right)\|\xhat_{k}\|^2 \notag\\
&\quad + 2L_{H}^2L_{G}^2(L_{H}+1)^2\left(\beta_{k}^2 + 2\gamma_{k} \right)\|\yhat_{k}\|^2.\label{lem:xhat:Ineq}
\end{align}
\end{lem}
\begin{lem}\label{lem:yhat}
Suppose that Assumptions \ref{assump:smooth:FH} -- \ref{assump:noise} hold. Let $\{x_{k},y_{k}\}$ be generated by \eqref{alg:xy}. Then we have for any $k\geq 0$
\begin{align}
\Eset\left[\|\yhat_{k+1}\|^2\,|\,\Qcal_{k}\right] 
&\leq \left(1-\mu_{G}\beta_{k}\right)\|\yhat_{k}\|^2 + \beta_{k}^2\Gamma_{22} + 2L_{G}^2(L_{H}+1)^2\beta_{k}^2\|\yhat_{k}\|^2 +  \frac{L_{G}\beta_{k}}{\mu_{G}}\|\xhat_{k}\|^2\notag\\ 
&\quad  +  L_{G}^2(L_{H}+2)\beta_{k}^2\|\xhat_{k}\|^2\label{lem:yhat:Ineq}.
\end{align}
\end{lem} 
Next, we show that $\Eset\left[\|\xhat_{k}\|^2 + \|\yhat_{k}\|^2\right]$ is bounded in the following lemma. 
\begin{lem}\label{lem:x+y_bound}
Suppose that Assumptions \ref{assump:smooth:FH} -- \ref{assump:noise} hold. Let $\{x_{k},y_{k}\}$ be generated by \eqref{alg:xy}. Suppose that the nonincreasing and nonnegative sequence of step sizes $\{\alpha_{k},\beta_{k}\}$ satisfy
\begin{align}
\begin{aligned}
&\frac{\beta_{k}}{\alpha_{k}} \leq \frac{\mu_{F}\mu_{G}}{2L_{G}},\quad \sum_{k=0}^{\infty}\alpha_{k} = \sum_{k=0}^{\infty}\beta_{k} = \infty,\\
&C_{1} \triangleq \max\left\{\sum_{t=0}^{\infty}\alpha_{k}^2 \,\,,\sum_{t=0}^{\infty}\beta_{k}^2 \,\,,\sum_{t=0}^{\infty}\frac{\beta_{k}^2}{\alpha_{k}}\right\} < \infty.
\end{aligned}\label{lem:x+y_bound:stepsize}
\end{align}
In addition, let $C$ be 
\begin{align}
&C_{2} \triangleq 2L_{G}^2(L_{H}+1)^2\sum_{t=0}^{\infty}\frac{\alpha_{k}^2}{2L_{G}^2} + 2(L_{H}^2+1)(\beta_{k}^2 + \gamma_{k}) < \infty.\label{lem:x+y_bound:C}   
\end{align} 
Then we obtain for all $k\geq 0$
\begin{align}
&\Eset\left[\|\xhat_{k}\|^2 + \|\yhat_{k}\|^2\right] \leq D \triangleq \frac{ C_{2}\Eset\left[\|\zhat_{0}\|^2\right] + C_{1}\exp\left(C_{2}\right)\left(3\Gamma_{22} + \Gamma_{11}\right)}{\exp(-C_{2})}\cdot   \label{lem:x+y_bound:Ineq}
\end{align}
\end{lem}
Finally, in our analysis we utilize the well-known almost supermartingale convergence result to establish the asymptotic convergence of our iterates \cite{Robbins1971}.
\begin{lem}[\cite{Robbins1971}]\label{lem:Martingale}
Let $\{w_{k}\},$ $\{v_{k}\}$, $\{\sigma_{k}\}$, and $\{\delta_{k}\}$ be non-negative sequences of random variables and satisfy 
\begin{align}
\begin{aligned}
&\mathbb{E}\Big[\,w_{k+1}\,|\,\Qcal_{k}\,\Big] \leq (1+\sigma_{k})w_{k} - v_{k} + \delta_{k}\\
&\sum_{k=0}^\infty \sigma_{k} < \infty \text{ a.s, }\quad \sum_{k=0}^\infty \delta_{k} < \infty \text{ a.s.}
\end{aligned}\label{lem:Martingale:Ineq1}
\end{align}
where $\Qcal_k = \{w(0),\ldots,w_{k}\}$. Then $\{w_{k}\}$ converges a.s., and  $\sum_{k=0}^{\infty} v_{k}<\infty$ a.s.
\end{lem}

\subsection{Convergence Results}\label{subsec:analysis}
In this section, we present the main results of this paper, which are the convergence properties of \eqref{alg:xy} under the assumptions stated in the previous section. We first study the asymptotic convergence of $(x_{k},y_{k})$ to $(x^*,y^*)$ in the following theorem. 
\begin{thm}\label{thm:asymptotic}
Suppose that Assumptions \ref{assump:smooth:FH} -- \ref{assump:noise} hold. Let $\{x_{k},y_{k}\}$ be generated by \eqref{alg:xy} with $\alpha_{k},\beta_{k}$ satisfying
\begin{align}
\begin{aligned}
&\frac{\beta_{0}}{\alpha_{0}} \leq \frac{\mu_{F}\mu_{G}}{2L_{G}},\quad \sum_{k=0}^{\infty}\alpha_{k} = \sum_{k=0}^{\infty}\beta_{k} = \infty\\  &\max\left\{\sum_{t=0}^{\infty}\alpha_{k}^2 \,\,,\sum_{t=0}^{\infty}\beta_{k}^2 \,\,,\sum_{t=0}^{\infty}\frac{\beta_{k}^2}{\alpha_{k}}\right\} < \infty.
\end{aligned}\label{thm:asymptotic:stepsizes}
\end{align}
Then we have
\begin{align}
\lim_{k\rightarrow\infty} \xhat_{k} = \lim_{k\rightarrow\infty}\yhat_{k} = 0\quad\text{a.s.} \label{thm:asymptotic:conv}   
\end{align}
\end{thm}
\begin{remark}
One can see from \eqref{thm:asymptotic:stepsizes} that $\beta_{k} \ll \alpha_{k}$. There are many choices of step sizes $\alpha_{k}$ and $\beta_{k}$ satisfying these conditions, e.g., let $1/2 < a < b\leq 1$, $2b-a>1$, and 
\begin{equation*}
\alpha_{k} = \frac{\alpha_{0}}{(k+2)^{a}},\quad \beta_{k} = \frac{\beta_{0}}{(k+2)^{b}}\cdot
\end{equation*}
In Theorem \ref{thm:rate} below, we choose $a = 2/3$ and $b = 1$ to have the best possible rate from our analysis for the convergence of the nonlinear two-time-scale SA \eqref{alg:xy}.
\end{remark}
\begin{proof}
For convenience we denote by 
\begin{equation*}
\zhat_{k} = \left[\begin{array}{c}
\xhat_{k}  \\
\yhat_{k} 
\end{array}\right].    
\end{equation*}
The main idea of our analysis is to use Lemmas \ref{lem:xhat} and \ref{lem:yhat} to show that $\|\zhat\|$ is an almost supermartingale sequence, therefore, it converges by Lemma \ref{lem:Martingale}. Indeed, adding \eqref{lem:xhat:Ineq} to \eqref{lem:yhat:Ineq} we obtain
\begin{align*}
\Eset\left[\|\zhat_{k+1}\|^2\,|\,\Qcal_{k}\right]
&\leq \left(1-\mu_{F}\alpha_{k} \right) \|\xhat_{k}\|^2 +  \beta_{k}^2\Gamma_{22} + L_{H}^2\gamma_{k}\Gamma_{22} + L_{H}^2\left(2L_{G}^2\beta_{k}^2 + \alpha_{k}^2+L_{G}^2\gamma_{k}\right)\|\xhat_{k}\|^2 \notag\\
&\quad + L_{H}^2L_{G}^2(L_{H}+1)^2\left(2\beta_{k}^2 +   \gamma_{k} \right)\|\yhat_{k}\|^2\notag\\
&\quad +  \left(1-\mu_{G}\beta_{k}\right)\|\yhat_{k}\|^2 + \beta_{k}^2\Gamma_{22} + 2L_{G}^2(L_{H}+1)^2\beta_{k}^2\|\yhat_{k}\|^2\notag\\ 
&\quad +  \frac{L_{G}\beta_{k}}{\mu_{G}}\|\xhat_{k}\|^2 +  L_{G}^2(L_{H}+2)\beta_{k}^2\|\xhat_{k}\|^2\allowdisplaybreaks\notag\\
&\leq \|\zhat_{k}\|^2 -\mu_{F}\alpha_{k}\|\xhat_{k}\|^2 - \mu_{G}\beta_{k}\|\yhat_{k}\|^2 +  \frac{L_{G}\beta_{k}}{\mu_{G}}\|\xhat_{k}\|^2 +  2\beta_{k}^2\Gamma_{22} +  \alpha_{k}^2\Gamma_{11} + L_{H}^2\gamma_{k}\Gamma_{22}\notag\\ 
&\quad + (L_{H}+1)^2\left(2L_{G}^2\beta_{k}^2 + \alpha_{k}^2+ L_{G}^2\gamma_{k}\right)\|\xhat_{k}\|^2 \notag\\
&\quad + (L_{H}^2+1)L_{G}^2(L_{H}+1)^2\left(2\beta_{k}^2 + \gamma_{k} \right)\|\yhat_{k}\|^2\allowdisplaybreaks\notag\\
&\leq \|\zhat_{k}\|^2 -\frac{\mu_{F}}{2}\alpha_{k}\|\xhat_{k}\|^2 - \mu_{G}\beta_{k}\|\yhat_{k}\|^2 +  2\beta_{k}^2\Gamma_{22} +  \alpha_{k}^2\Gamma_{11} + L_{H}^2\gamma_{k}\Gamma_{22}\notag\\
&\quad + L_{G}^2(L_{H}+1)^4\left(2L_{G}^2\beta_{k}^2 + \alpha_{k}^2 + \gamma_{k}\right)\|\zhat_{k}\|^2,
\end{align*}
where the last inequality we use \eqref{thm:asymptotic:stepsizes} to have
\begin{align*}
-\mu_{F}\alpha_{k} + \frac{L_{G}\beta_{k}}{\mu_{G}} \leq -\frac{\mu_{F}\alpha_{k}}{2}\cdot    
\end{align*}
Since $\alpha_{k},\beta_{k}$ satisfy \eqref{thm:asymptotic:stepsizes}, the relation above satisfies all the conditions in \eqref{lem:Martingale:Ineq1}, where $w_{k} = \|\zhat_{k}\|^2$. We, therefore, obtain
\begin{align*}
\left\{\begin{array}{l}
\|\zhat_{k}\|^2 = \|\xhat_{k}\|^2 + \|\yhat_{k}\|^2\; \text{ converges a.s.}  \\
    \sum_{k=0}^{\infty}  \left(\frac{\mu_{F}}{2}\alpha_{k}\|\xhat_{k}\|^2 + \mu_{G}\beta_{k}\|\yhat_{k}\|^2\right) < \infty\; \text{ a.s.}
\end{array} \right.   
\end{align*}
Since $\sum_{k=0}^{\infty} \alpha_{k} = \sum_{k=0}^{\infty} \beta_{k} = \infty$, these  relations gives \eqref{thm:asymptotic:conv}, which concludes our proof.  
\end{proof}
We next study the convergence rate of the nonlinear two-time-scale SA, where we provide a finite-time analysis for the mean square error generated by \eqref{alg:xy} to zero. The main idea is to utilize Lemmas \ref{lem:xhat}--\ref{lem:x+y_bound} and the following Lyapunov equation representing the coupling between the two iterates  
\begin{align}
V(\xhat_{k},\yhat_{k}) &= \|\yhat_{k}\|^2 + \frac{2L_{G}}{\mu_{F}\mu_{G}}\frac{\beta_{k}}{\alpha_{k}} \|\xhat_{k}\|^2\notag\\ 
&= \|y_{k} - y^{\star}\|^2 + \frac{2L_{G}}{\mu_{F}\mu_{G}}\frac{\beta_{k}}{\alpha_{k}} \|x_{k} - H(y_{k})\|^2 .\label{notation:V}    
\end{align}
The convergence rate of $V_{k}$ to zero in expectation is formally stated in the following theorem. 

\begin{thm}\label{thm:rate}
Suppose that Assumptions \ref{assump:smooth:FH} -- \ref{assump:noise} hold. Let $\{x_{k},y_{k}\}$ be generated by \eqref{alg:xy} with  $\alpha_{k},\beta_{k}$ satisfying
\begin{align}
\begin{aligned}
&\frac{\beta_{0}}{\alpha_{0}} \leq \min\left\{\frac{\mu_{F}\mu_{G}}{2L_{G}}\,,\,\frac{\mu_{F}}{2\mu_{G}}\right\},\quad \beta_{0} \geq \frac{2}{\mu_{G}},\\
&\alpha_{k} = \frac{\alpha_{0}}{(k+2)^{2/3}},\quad \beta_{k} = \frac{\beta_{0}}{(k+2)}.
\end{aligned}\label{thm:rate:stepsizes}
\end{align}
Let $D$ be defined in \eqref{lem:x+y_bound:Ineq}. Then we have for all $k\geq0$
\begin{align}
\Eset\left[V(\xhat_{k},\yhat_{k})\right]
&\leq \frac{\Eset[V(\xhat_{0},\yhat_{0})]}{(k+1)^2} + \frac{(2L_{G} + \mu_{F}\mu_{G})\beta_{0}^2\Gamma_{22} }{\mu_{F}\mu_{G}(k+1)}\notag\\ 
&\quad + \frac{4DL_{G}^2(L_{H}^2L_{G}+\mu_{F}\mu_{G})(L_{H}+1)^2\beta_{0}^2}{\mu_{F}\mu_{G}(k+1)}\notag\\
&\quad +  \frac{2L_{G}\alpha_{0}\beta_{0}\Gamma_{11}}{\mu_{F}\mu_{G}(k+1)^{2/3}} + \frac{2L_{G}L_{H}^2(1+L_{F}\alpha_{0})^2\beta_{0}^3\Gamma_{22}}{\mu_{F}^2\mu_{G}\alpha_{0}^2(k+1)^{2/3}}\notag\\
&\quad + \frac{2DL_{H}^2L_{G}(L_{H}+1)^2\beta_{0}(\alpha_{0}+ \mu_{F}\mu_{G}^2)}{\mu_{F}\mu_{G}(k+1)^{2/3}}\cdot    \label{thm:rate:ineq}
\end{align}
\end{thm}

\begin{proof}
For convenience, let $\omega_{k}$ be
\begin{equation*}
\omega_{k} = \frac{2L_{G}}{\mu_{F}\mu_{G}}\frac{\beta_{k}}{\alpha_{k}},    
\end{equation*}
which is nonincreasing due to \eqref{thm:rate:stepsizes}. The conditions in \eqref{thm:rate:stepsizes} obviously satisfy \eqref{lem:x+y_bound:stepsize}. Taking the expectation on both sides of \eqref{lem:xhat:Ineq} and multiplying both sides of \eqref{lem:xhat:Ineq} by $\omega_{k}$ we obtain
\begin{align*}
\omega_{k+1}\Eset\left[\|\xhat_{k+1}\|^2\right]
&\leq \omega_{k}\Big(1-\mu_{F}\alpha_{k} \Big) \Eset\left[\|\xhat_{k}\|^2\right]\notag\\ 
&\quad +  \frac{2L_{G}}{\mu_{F}\mu_{G}}\Big(\beta_{k}^2\Gamma_{22} +  \alpha_{k}\beta_{k}\Gamma_{11} + \frac{L_{H}^2\gamma_{k}\beta_{k}}{\alpha_{k}}\Gamma_{22}\Big)\notag\\ 
&\quad + \frac{2L_{G}L_{H}^2}{\mu_{F}\mu_{G}}\left(2L_{G}^2\beta_{k}^2 + \alpha_{k}\beta_{k} + \frac{L_{G}^2\gamma_{k}\beta_{k}}{\alpha_{k}}\right)\Eset\left[\|\xhat_{k}\|^2\right] \notag\\
&\quad + \frac{2L_{H}^2L_{G}^3(L_{H}+1)^2}{\mu_{F}\mu_{G}}\left(2\beta_{k}^2 +   \frac{\gamma_{k}\beta_{k}}{\alpha_{k}} \right)\Eset\left[\|\yhat_{k}\|^2\right]\notag\\
&= (1-\mu_{G}\beta_{k})\omega_{k} \Eset\left[\|\xhat_{k}\|^2\right] - \left(\mu_{F}\omega_{k}\alpha_{k} - \mu_{G}\beta_{k}\omega_{k}\right)\Eset\left[\|\xhat_{k}\|^2\right]\notag\\ 
&\quad +  \frac{2L_{G}}{\mu_{F}\mu_{G}}\left(\beta_{k}^2\Gamma_{22} +  \alpha_{k}\beta_{k}\Gamma_{11} + \frac{L_{H}^2\gamma_{k}\beta_{k}}{\alpha_{k}}\Gamma_{22}\right)\notag\\ 
&\quad + \frac{2L_{G}L_{H}^2}{\mu_{F}\mu_{G}}\left(2L_{G}^2\beta_{k}^2 + \alpha_{k}\beta_{k} + \frac{L_{G}^2\gamma_{k}\beta_{k}}{\alpha_{k}}\right)\Eset\left[\|\xhat_{k}\|^2\right] \notag\\
&\quad + \frac{2L_{H}^2L_{G}^3(L_{H}+1)^2}{\mu_{F}\mu_{G}}\left(2\beta_{k}^2 +   \frac{\gamma_{k}\beta_{k}}{\alpha_{k}} \right)\Eset\left[\|\yhat_{k}\|^2\right].
\end{align*}
Taking the expectation on both sides of  \eqref{lem:yhat:Ineq} and adding to the preceding relation yields 
\begin{align*}
\Eset[V(\xhat_{k+1},\yhat_{k+1})] &\leq (1-\mu_{G}\beta_{k})\Eset[V(\xhat_{k},\yhat_{k})]\notag\\ 
&\quad - \left(\mu_{F}\omega_{k}\alpha_{k} - \mu_{G}\beta_{k}\omega_{k} - \frac{L_{G}\beta_{k}}{\mu_{G}}\right)\Eset\left[\|\xhat_{k}\|^2\right]\notag\\ 
&\quad +  \frac{2L_{G}}{\mu_{F}\mu_{G}}\left(\beta_{k}^2\Gamma_{22} +  \alpha_{k}\beta_{k}\Gamma_{11} + \frac{L_{H}^2\gamma_{k}\beta_{k}}{\alpha_{k}}\Gamma_{22}\right)\notag\\ 
&\quad + \frac{2L_{G}L_{H}^2}{\mu_{F}\mu_{G}}\left(2L_{G}^2\beta_{k}^2 + \alpha_{k}\beta_{k} + \frac{L_{G}^2\gamma_{k}\beta_{k}}{\alpha_{k}}\right)\Eset\left[\|\xhat_{k}\|^2\right] \notag\\
&\quad + \frac{2L_{H}^2L_{G}^3(L_{H}+1)^2}{\mu_{F}\mu_{G}}\left(2\beta_{k}^2 +   \frac{\gamma_{k}\beta_{k}}{\alpha_{k}}\right)\Eset\left[\|\yhat_{k}\|^2\right]\notag\\
&\quad + \beta_{k}^2\Gamma_{22} + 2L_{G}^2(L_{H}+1)^2\beta_{k}^2\Eset\left[\|\yhat_{k}\|^2\right] +  L_{G}^2(L_{H}+2)\beta_{k}^2\Eset\left[\|\xhat_{k}\|^2\right],
\end{align*}
which by using \eqref{thm:rate:stepsizes} to have
\begin{align*}
 &\mu_{F}\omega_{k}\alpha_{k} - \mu_{G}\beta_{k}\omega_{k} - \frac{L_{G}\beta_{k}}{\mu_{G}} = \frac{2L_{G}\beta_{k}}{\mu_{G}} - \frac{2L_{G}}{\mu_{F}}\frac{\beta_{k}^2}{\alpha_{k}} - \frac{L_{G}\beta_{k}}{\mu_{G}} \geq 0,
\end{align*}
we obtain 
\begin{align}
\Eset[V(\xhat_{k+1},\yhat_{k+1})] &\leq (1-\mu_{G}\beta_{k})\Eset[V(\xhat_{k},\yhat_{k})] +  2L_{G}^2(L_{H}+1)^2\beta_{k}^2\Eset\left[\|\zhat_{k}\|^2\right]\notag\\ 
& +  \frac{2L_{G}}{\mu_{F}\mu_{G}}\left(\beta_{k}^2\Gamma_{22} +  \alpha_{k}\beta_{k}\Gamma_{11} + \frac{L_{H}^2\gamma_{k}\beta_{k}}{\alpha_{k}}\Gamma_{22}\right) + \beta_{k}^2\Gamma_{22}\notag\\ 
& + \frac{2L_{H}^2L_{G}^3(L_{H}+1)^2}{\mu_{F}\mu_{G}}\left(2\beta_{k}^2 + \frac{\alpha_{k}\beta_{k}}{L_{G}^2} +   \frac{\gamma_{k}\beta_{k}}{\alpha_{k}} \right)\Eset\left[\|\zhat_{k}\|^2\right]
.\label{thm:rate:eq1}
\end{align}
Using $\gamma_{k}$ in \eqref{notation:gamma} and $\alpha_{k},\beta_{k}$ in \eqref{thm:rate:stepsizes} we have
\begin{align*}
&1.\; (k+2)^2(1-\mu_{G}\beta_{k})= (k+2)^2\left(1 - \frac{\mu_{G}\beta_{0}}{k+2}\right)\\ &\hspace{2.5cm}\leq (k+2)^2\left(1 - \frac{2}{k+2}\right) \leq (k+1)^2.\\
&2.\; (k+2)^2\alpha_{k}\beta_{k} = \alpha_{0}\beta_{0}(k+2)^{1/3}.\\
&3.\; \frac{\gamma_{k} \beta_{k}}{\alpha_{k}}(k+2)^2 \leq \frac{(1+L_{F}\alpha_{0})^2\beta_{0}^3}{\mu_{F}\alpha_{0}^2}(k+2)^{1/3}.
\end{align*}
Thus, multiplying both sides of \eqref{thm:rate:eq1} by $(k+2)^2$ and using \eqref{lem:x+y_bound:Ineq} and \eqref{thm:rate:stepsizes} we obtain 
\begin{align*}
(k+2)^2\Eset[V(\xhat_{k+1},\yhat_{k+1})]
&\leq (k+1)^2\Eset[V(\xhat_{k},\yhat_{k})] + 2L_{G}^2(L_{H}+1)^2\beta_{0}^2\Eset\left[\|\zhat_{k}\|^2\right] \notag\\ 
&\quad +  \frac{2L_{G}}{\mu_{F}\mu_{G}}\left(\beta_{0}^2\Gamma_{22} +  \alpha_{0}\beta_{0}\Gamma_{11}(k+2)^{1/3}\right) + \beta_{0}^2\Gamma_{22} 
\notag\\ 
&\quad + \frac{2L_{G}L_{H}^2(1+L_{F}\alpha_{0})^2\beta_{0}^3\Gamma_{22}}{\mu_{F}^2\mu_{G}\alpha_{0}^2}(k+2)^{1/3}\notag\\
&\quad + \frac{4L_{H}^2L_{G}^3(L_{H}+1)^2\beta_{0}^2}{\mu_{F}\mu_{G}}\Eset\left[\|\zhat_{k}\|^2\right]\notag\\
&\quad + \frac{2L_{H}^2L_{G}(L_{H}+1)^2\alpha_{0}\beta_{0}}{\mu_{F}\mu_{G}}(k+2)^{1/3}\Eset\left[\|\zhat_{k}\|^2\right]\notag\\
&\quad + \frac{2L_{H}^2L_{G}^3(L_{H}+1)^2(1+L_{F}\alpha_{0})^2\beta_{0}^3}{\mu_{F}^2\mu_{G}\alpha_{0}^2}(k+2)^{1/3}\Eset\left[\|\zhat_{k}\|^2\right]\notag\\
&\stackrel{\eqref{lem:x+y_bound:Ineq}}{\leq} (k+1)^2\Eset[V(\xhat_{k},\yhat_{k})] +\frac{(2L_{G} + \mu_{F}\mu_{G})\beta_{0}^2\Gamma_{22} }{\mu_{F}\mu_{G}} +  \frac{2L_{G}\alpha_{0}\beta_{0}\Gamma_{11}}{\mu_{F}\mu_{G}}(k+2)^{1/3} 
\notag\\ 
&\quad + \frac{2L_{G}L_{H}^2(1+L_{F}\alpha_{0})^2\beta_{0}^3\Gamma_{22}}{\mu_{F}^2\mu_{G}\alpha_{0}^2}(k+2)^{1/3}\notag\\
&\quad + \frac{4DL_{G}^2(L_{H}^2L_{G}+\mu_{F}\mu_{G})(L_{H}+1)^2\beta_{0}^2}{\mu_{F}\mu_{G}}\notag\\
&\quad + \frac{2DL_{H}^2L_{G}(L_{H}+1)^2\beta_{0}(\alpha_{0}+ \mu_{F}\mu_{G}^2)}{\mu_{F}\mu_{G}}(k+2)^{1/3},
\end{align*}
where the last inequality we use $\beta_{0}/\alpha_{0} \leq \mu_{F}\mu_{G}/(2L_{G})$ in the last term. Summing up both sides over $k = 0,\ldots,K$ for some $K>0$ we obtain
\begin{align*}
(K+2)^2\Eset[V(\xhat_{K+1},\yhat_{K+1})]
&\leq \Eset[V(\xhat_{0},\yhat_{0})] +\frac{(2L_{G} + \mu_{F}\mu_{G})\beta_{0}^2\Gamma_{22} }{\mu_{F}\mu_{G}}(K+1)\notag\\ 
&\quad +  \frac{2L_{G}\alpha_{0}\beta_{0}\Gamma_{11}}{\mu_{F}\mu_{G}}(K+2)^{4/3} 
\notag\\ 
&\quad + \frac{2L_{G}L_{H}^2(1+L_{F}\alpha_{0})^2\beta_{0}^3\Gamma_{22}}{\mu_{F}^2\mu_{G}\alpha_{0}^2}(K+2)^{4/3}\notag\\
&\quad + \frac{4DL_{G}^2(L_{H}^2L_{G}+\mu_{F}\mu_{G})(L_{H}+1)^2\beta_{0}^2}{\mu_{F}\mu_{G}}(K+1)\notag\\
&\quad + \frac{2DL_{H}^2L_{G}(L_{H}+1)^2\beta_{0}(\alpha_{0}+ \mu_{F}\mu_{G}^2)}{\mu_{F}\mu_{G}}(K+2)^{4/3},    
\end{align*}
where we use the integral test to have for all $k\geq 0$ 
\begin{align*}
&\sum_{k=0}^{K}(k+2)^{1/3} \leq \frac{1}{2^{1/3}} + \int_{t=0}^{K}(t+2)^{1/3}dt\notag\\ &= \frac{1}{2^{1/3}} + \frac{3}{4}\left((K+2)^{4/3} - \frac{1}{2^{4/3}}\right) \leq (K+2)^{4/3}.
\end{align*}
Diving both sides of the equation above by $(K+2)^2$ immediately gives \eqref{thm:rate:ineq}.
\end{proof}

\begin{remark}
First, we note that this convergence rate is the same as the one in the linear case studied in \cite{Doan_two_time_SA2019,DoanR2019} and in \cite{DalalTSM2018} with convergence in high probability. On the other hand, the recent work by \cite{Kaledin_two_time_SA2020} and \cite{Dalal_Szorenyi_Thoppe_2020} improve this rate to $1/k$ by further exploiting the linear structure of the two-time-scale SA. As a result it is not obvious whether one can apply the techniques in \cite{Kaledin_two_time_SA2020,Dalal_Szorenyi_Thoppe_2020} to improve the rates of the nonlinear two-time-scale SA studied in this paper.

Second, as observed from Eq.\ \eqref{thm:rate:eq1}, the two terms decide the rate of the two-time-scale {\sf SA} are $\beta_{k}\alpha_{k}$ and $\beta_{k}^3/\alpha_{k}^2$, which also characterize the coupling between the two iterates. Our choice of $\alpha$ and $\beta_{k}$ in the theorem is to balance these two terms, i.e., we want to achieve
\begin{align*}
\alpha_{k}\beta_{k} = \frac{\beta_{k}^3}{\alpha_{k}^2} \Rightarrow \alpha_{k}^{3} = \beta_{k}^2.      
\end{align*}
One can choose different step sizes, which will result into different decaying rates for different components in \eqref{thm:rate:ineq}. These step sizes, however, need to satisfy the conditions in \eqref{thm:asymptotic:stepsizes}. 
\end{remark}

\section{Proofs of Lemmas \ref{lem:xhat}--\ref{lem:x+y_bound}}\label{sec:preliminaries}
In this section, we provide the analysis for the results presented in Lemmas \ref{lem:xhat}--\ref{lem:x+y_bound}. 
\subsection{Proof of Lemma \ref{lem:xhat}}
Recall that $\xhat = x-H(y)$. We consider 
\begin{align*}
\xhat_{k+1} = x_{k+1} - H(y_{k+1})
&= x_{k} - \alpha_{k}F(x_{k},y_{k}) - \alpha_{k}\xi_{k} - H(y_{k+1})\notag\\
&= \xhat_{k} -  \alpha_{k}F(x_{k},y_{k}) - \alpha_{k}\xi_{k}  + H(y_{k}) - H(y_{k+1}),
\end{align*}
which implies
\begin{align}
\|\xhat_{k+1}\|^2 &= \big\|\xhat_{k}  - \alpha_{k}F(x_{k},y_{k}) - \alpha_{k}\xi_{k} + H(y_{k}) - H(y_{k+1})\big\|^2\notag\\
&= \|\xhat_{k}  - \alpha_{k}F(x_{k},y_{k})\|^2 + \left\|H(y_{k}) - H(y_{k+1}) - \alpha_{k}\xi_{k}\right\|^2\notag\\
&\quad + 2(\xhat_{k}-F(x_{k},y_{k}))^T(H(y_{k}) - H(y_{k+1}))  - 2\alpha_{k}^2\left(\xhat_{k} - \alpha_{k}F(x_{k},y_{k})\right)^T\xi_{k}. \label{lem:xhat:Eq1}
\end{align}
We next analyze each term on the right-hand side of \eqref{lem:xhat:Eq1}. First, using $F(H(y_{k}),y_{k}) = 0$ we have 
\begin{align}
\|\xhat_{k}  - \alpha_{k}F(x_{k},y_{k})\|^2 
&= \|\xhat_{k} \|^2 - 2\alpha_{k}(x_{k} - H(y_{k}) )^TF(x_{k},y_{k})  + \| \alpha_{k}F(x_{k},y_{k})\|^2\notag\\
&= \|\xhat_{k}\|^2 - 2\alpha_{k}\xhat_{k} ^T\left(F(x_{k},y_{k}) - F(H(y_{k}),y_{k})\right)\notag\\ 
&\quad  + \alpha_{k}^2\| F(x_{k},y_{k}) - F(H(y_{k}),y_{k})\|^2\notag\\
&\leq \|\xhat_{k}\|^2 - 2\mu_{F}\alpha_{k}\|\xhat_{k} \|^2 + L_{H}^2\alpha_{k}^2\|\xhat_{k} \|^2\notag\\ 
&= \left(1-2\mu_{F}\alpha_{k} + L_{H}^2\alpha_{k}^2\right) \|\xhat_{k}\|^2,\label{lem:xhat:Eq1a}
\end{align}
where the inequality is due to the strong monotone and Lispchitz continuity of $F$, i.e., Eqs.\ \eqref{assump:sm:F:ineq} and \eqref{assump:smooth:FH:ineqF}, respectively. Recall that  
\begin{align*}
\Qcal_{k} = \{x_{0},y_{0},\xi_{0},\psi_{0},\xi_{1},\psi_{1},\ldots,\xi_{k-1},\psi_{k-1}\}.
\end{align*}
We next take the conditional expectation of the second term on the right-hand side of \eqref{lem:xhat:Eq1} w.r.t $\Qcal_{k}$ and using Assumption \ref{assump:noise} to have
\begin{align}
&\Eset\left[\left\|H(y_{k}) - H(y_{k+1}) - \alpha_{k}\xi_{k}\right\|^2\,|\,\Qcal_{k}\right]\notag\\
&= \Eset\left[\left\|H(y_{k}) - H(y_{k+1})\right\|^2\,|\,\Qcal_{k} \right]  + \alpha_{k}^2\Eset\left[\left\|\xi_{k}\right\|^2\,|\,\Qcal_{k}\right]\notag\\
&\leq L_{H}^2\Eset\left[\left\|\beta_{k}G(x_{k},y_{k}) - \beta_{k}\psi_{k}\right\|^2\,|\,\Qcal_{k} \right] + \alpha_{k}^2\Eset\left[\left\|\xi_{k}\right\|^2\,|\,\Qcal_{k}\right]\notag\\
&= L_{H}^2\beta_{k}^2\left\|G(x_{k},y_{k})\right\|^2 + L_{H}^2\beta_{k}^2\Eset\left[\|\psi_{k}\|^2\,|\,\Qcal_{k}\right]  +  \alpha_{k}^2\Eset\left[\left\|\xi_{k}\right\|^2\,|\,\Qcal_{k}\right]\notag\\
&= L_{H}^2\beta_{k}^2\left\|G(x_{k},y_{k})\right\|^2 +  \beta_{k}^2\Gamma_{22} +  \alpha_{k}^2\Gamma_{11}\notag\\
&\leq 2L_{H}^2L_{G}^2\beta_{k}^2\|\xhat_{k}\|^2 + 2L_{H}^2L_{G}^2(L_{H}+1)^2\beta_{k}^2\|\yhat_{k}\|^2     +  \beta_{k}^2\Gamma_{22} +  \alpha_{k}^2\Gamma_{11}, \label{lem:xhat:Eq1b}
\end{align}
where the last inequality we use  $G(H(y^{\star}),y^{\star}) = 0$ and the Lipschitz continuity of $G$ and $H$ to obtain 
\begin{align}
\|G(x_{k},y_{k})\|^2
 &\leq 2\|G(x_{k},y_{k}) - G(H(y_{k}),y_{k})\|^2 +  2\|G(H(y_{k}),y_{k}) - G(H(y^{\star}),y^{\star})\|^2\notag\\
&\leq 2L_{G}^2\|\xhat_{k}\|^2 + 2L_{G}^2(\|H(y_{k}) - H(y^{\star})\| + \|y_{k}-y^{\star}\|^2)\notag\\
&\leq 2L_{G}^2\|\xhat_{k}\|^2 + 2L_{G}^2(L_{H}+1)^2\|\yhat_{k}\|^2.\label{lem:xhat:Eq1b1}
\end{align}
Finally, we analyze the third term by using \eqref{lem:xhat:Eq1a} and the preceding relation
\begin{align*}
&\left(\xhat_{k} - \alpha_{k}F(x_{k},y_{k})\right)^T(H(y_{k}) - H(y_{k+1}))
\leq \|\xhat_{k} - \alpha_{k}F(x_{k},y_{k})\|\|H(y_{k}) - H(y_{k+1})\|\notag\\
&\qquad\leq L_{H}\beta_{k}(\|\xhat_{k}\| + \alpha_{k}\|F(x_{k},y_{k})\|)\|G(x_{k},y_{k}) + \psi_{k}\|\notag\\ 
&\qquad= L_{H}\beta_{k}(\|\xhat_{k}\| + \alpha_{k}\|F(x_{k},y_{k}) - F(H(y_{k}),y_{k})\|) \|G(x_{k},y_{k}) + \psi_{k}\|\notag\\ 
&\qquad\leq L_{H}(1+L_{F}\alpha_{k})\beta_{k}\|\xhat_{k}\|\left(\|G(x_{k},y_{k})\| + \|\psi_{k}\|\right), 
\end{align*}
where the equality is due to $F(H(y_{k}),y_{k}) = 0$. Applying the second inequality in Eq.\ \eqref{lem:xhat:Eq1b1} to the preceding relation yields  
\begin{align}
&\left(\xhat_{k} - \alpha_{k}F(x_{k},y_{k})\right)^T(H(y_{k}) - H(y_{k+1}))\notag\\
&\qquad\leq L_{H}(1+L_{F}\alpha_{k})\beta_{k}\|\xhat_{k}\|\left(L_{G}\|\xhat_{k}\| + L_{G}(L_{H}+1)\|\yhat_{k}\| + \|\psi_{k}\|\right)\notag\\
&\qquad\leq \frac{\mu_{F}}{2}\alpha_{k}\|\xhat_{k}\|^2 + \gamma_{k}\left(L_{G}\|\xhat_{k}\| + L_{G}(L_{H}+1)\|\yhat_{k}\| + \|\psi_{k}\|\right)^2\notag\\
&\qquad\leq \frac{\mu_{F}}{2}\alpha_{k}\|\xhat_{k}\|^2 + 2L_{H}^2\gamma_{k} \left(L_{G}^2\|\xhat_{k}\|^2 + L_{G}^2(L_{H}+1)^2\|\yhat_{k}\|^2 + \|\psi_{k}\|^2\right),    \label{lem:xhat:Eq1c}
\end{align}
where the second inequality is due to the Cauchy-Schwarz inequality $2ab\leq \eta a^2 + 1/\eta b^2$, $\forall \eta >0$, and  $\gamma_{k}$ is defined in \eqref{notation:gamma}. Taking the conditional expectation on both sides of \eqref{lem:xhat:Eq1} w.r.t $\Qcal_{k}$ and using \eqref{lem:xhat:Eq1a}--\eqref{lem:xhat:Eq1c} and Assumption \ref{assump:noise} we obtain
\begin{align*}
\Eset\left[\|\xhat_{k+1}\|^2\,|\,\Qcal_{k}\right]
&= \left(1-2\mu_{F}\alpha_{k} \right) \|\xhat_{k}\|^2 + L_{H}^2(2L_{G}^2\beta_{k}^2 + \alpha_{k}^2)\|\xhat_{k}\|^2 \notag\\
&\quad + 2L_{H}^2L_{G}^2(L_{H}+1)^2\left(\beta_{k}^2 + 2\gamma_{k} \right)\|\yhat_{k}\|^2\notag\\
&\quad +  \beta_{k}^2\Gamma_{22} +  \alpha_{k}^2\Gamma_{11} + 4L_{H}^2\gamma_{k}\Gamma_{22} \notag\\
&\quad + \mu_{F}\alpha_{k}\|\xhat_{k}\|^2 + 4L_{H}^2L_{G}^2\gamma_{k}\|\xhat_{k}\|^2 \notag\\
&\leq \left(1-\mu_{F}\alpha_{k} \right) \|\xhat_{k}\|^2 +  \beta_{k}^2\Gamma_{22} +  \alpha_{k}^2\Gamma_{11} + 4L_{H}^2\gamma_{k}\Gamma_{22}\notag\\ 
&\quad + L_{H}^2\left(2L_{G}^2\beta_{k}^2 + \alpha_{k}^2+ 4L_{G}^2\gamma_{k}\right)\|\xhat_{k}\|^2 \notag\\
&\quad + 2L_{H}^2L_{G}^2(L_{H}+1)^2\left(\beta_{k}^2 + 2\gamma_{k} \right)\|\yhat_{k}\|^2,  
\end{align*}
which concludes our proof. 
\subsection{Proof of Lemma \ref{lem:yhat}}
Recall that $\yhat = y - y^*$. Using \eqref{alg:xy} we consider
\begin{align*}
\yhat_{k+1} &= y_{k+1} - y^{\star} = y_{k} - y^{\star} - \beta_{k}G(x_{k},y_{k})  - \beta_{k}\psi_{k} \notag\\
&= \yhat_{k} - \beta_{k}G(H(y_{k}),y_{k}) + \beta_{k} \left(G(H(y_{k}),y_{k}) - G(x_{k},y_{k})\right)  - \beta_{k}\psi_{k},
\end{align*}
which implies that
\begin{align}
\|\yhat_{k+1}\|^2 &= \|\yhat_{k} - \beta_{k}G(H(y_{k}),y_{k})\|^2 + \left\|\beta_{k} \left(G(H(y_{k}),y_{k}) - G(x_{k},y_{k})\right)  - \beta_{k}\psi_{k}\right\|^2\notag\\
&\quad + 2\beta_{k}(\yhat_{k} - \beta_{k}G(H(y_{k}),y_{k}))^T\left(G(H(y_{k}),y_{k}) - G(x_{k},y_{k})\right)\notag\\
&\quad - 2\beta_{k}(\yhat_{k} - \beta_{k}G(H(y_{k}),y_{k}))^T\psi_{k}. \label{lem:yhat:Eq1}
\end{align}
We next analyze each term on the right-hand side of \eqref{lem:yhat:Eq1}. First, using ${G(H(y^{\star}),y^{\star}) = 0}$, \eqref{assump:G:sm}, \eqref{assump:smooth:FH:ineqH} and \eqref{assump:G:smooth}  we consider the first term 
\begin{align}
\left\|\yhat_{k} - \beta_{k}G(H(y_{k}),y_{k})\right\|^2 &= \left\|\yhat_{k}\right\|^2 - 2\beta_{k}\yhat_{k}^TG(H(y_{k}),y_{k}) + \beta_{k}^2\left\|G(H(y_{k}),y_{k})\right\|^2\notag\\ 
&\leq \|\yhat_{k}\|^2 - 2\mu_{G}\beta_{k}\|\yhat_{k}\|^2  + \beta_{k}^2 \|G(H(y_{k}),y_{k}) - G(H(y^{\star}),y^{\star})\|^2\notag\\
&\leq (1- 2\mu_{G}\beta_{k})\|\yhat_{k}\|^2  + 2\beta_{k}^2 \|G(H(y_{k}),y_{k}) - G(H(y_{k}),y^{\star})\|^2\notag\\
&\quad + 2\beta_{k}^2 \| G(H(y_{k}),y^{\star}) - G(H(y^{\star}),y^{\star})\|^2\notag\\
&\leq (1- 2\mu_{G}\beta_{k})\|\yhat_{k}\|^2  + 2\beta_{k}^2 \|G(H(y_{k}),y_{k}) - G(H(y_{k}),y^{\star})\|^2\notag\\
&\quad +2\beta_{k}^2\| G(H(y_{k}),y^{\star}) - G(H(y^{\star}),y^{\star})\|^2\notag\\
&\leq (1- 2\mu_{G}\beta_{k})\|\yhat_{k}\|^2 + 2L_{G}^2\beta_{k}^2 \left(\|H(y_{k}) - H(y^{\star})\|^2 + \|\yhat_{k}\|^2\right)\notag\\
&\leq \left(1-2\mu_{G}\beta_{k} + 2(L_{H}^2+1)L_{G}^2\beta_{k}^2\right)\|\yhat_{k}\|^2.\label{lem:yhat:Eq1a}
\end{align}
Next, taking the conditional expectation of the second term on the right-hand side of \eqref{lem:yhat:Eq1} w.r.t $\Qcal_{k}$ and using Assumption \ref{assump:noise} and \eqref{assump:G:smooth} we have
\begin{align}
&\Eset\left[\left\|\beta_{k} \left(G(H(y_{k}),y_{k}) - G(x_{k},y_{k})\right)  - \beta_{k}\psi_{k}\right\|^2\;|\;\Qcal_{k}\right]\notag\\
&= \beta_{k}^2 \left\| G(H(y_{k}),y_{k}) - G(x_{k},y_{k})\right\|^2 +\beta_{k}^2 \Eset\left[\left\|\psi_{k}\right\|^2\;|\;\Qcal_{k}\right] \leq L_{G}^2\beta_{k}^2\|\xhat_{k}\|^2 + \beta_{k}^2\Gamma_{22}. 
\label{lem:yhat:Eq1b}
\end{align}
Finally, using Assumption \ref{assump:noise} and \eqref{lem:yhat:Eq1a} we consider the third term on the right-hand side of \eqref{lem:yhat:Eq1} 
\begin{align}
&2\beta_{k}(\yhat_{k} - \beta_{k}G(H(y_{k}),y_{k}))^T(G(H(y_{k}),y_{k}) - G(x_{k},y_{k})) \notag\\
&\leq 2\beta_{k}\left(\|\yhat_{k}\| + \beta_{k}\|G(H(y_{k}),y_{k})\|\right)L_{G}\|\xhat_{k}\|\notag\\
&= 2L_{G}\|\xhat_{k}\|\beta_{k}\Big(\|\yhat_{k}\| + \beta_{k}\|G(H(y_{k}),y_{k})- G(H(y_{k}),y^{\star}) + G(H(y_{k}),y^{\star}) - G(H(y^{\star}),y^{\star}) \|\Big)\notag\\
&\leq 2L_{G}(1+(L_{H}+1)L_{G}\beta_{k})\beta_{k}\|\yhat_{k}\|\|\xhat_{k}\|\notag\\
&= 2L_{G}\beta_{k}\|\yhat_{k}\|\|\xhat_{k}\| + 2L_{G}^2(L_{H}+1)\beta_{k}^2\|\yhat_{k}\|\|\xhat_{k}\|\notag\\
&\leq \mu_{G}\beta_{k}\|\yhat_{k}\|^2 + \frac{L_{G}\beta_{k}}{\mu_{G}}\|\xhat_{k}\|^2 +  L_{G}^2(L_{H}+1)\beta_{k}^2(\|\yhat_{k}\|^2 + \|\xhat_{k}\|^2),\label{lem:yhat:Eq1c}
\end{align}
where the last inequality is due to the Cauchy-Schwarz inequality. Thus, taking the conditional expectation of \eqref{lem:xhat:Eq1} w.r.t $\Qcal_{k}$ and using \eqref{lem:yhat:Eq1a}--\eqref{lem:yhat:Eq1c} yields
\begin{align*}
\Eset\left[\|\yhat_{k+1}\|^2\,|\,\Qcal_{k}\right]
&\leq \left(1-2\mu_{G}\beta_{k} + 2(L_{H}^2+1)L_{G}^2\beta_{k}^2\right)\|\yhat_{k}\|^2 +  L_{G}^2\beta_{k}^2\|\xhat_{k}\|^2 \notag\\
&\quad + \beta_{k}^2\Gamma_{22} + \mu_{G}\beta_{k}\|\yhat_{k}\|^2+ \frac{L_{G}\beta_{k}}{\mu_{G}}\|\xhat_{k}\|^2\notag\\ 
&\quad  +  L_{G}^2(L_{H}+1)\beta_{k}^2(\|\yhat_{k}\|^2 + \|\xhat_{k}\|^2)\notag\\
&= \left(1-\mu_{G}\beta_{k}\right)\|\yhat_{k}\|^2 + \beta_{k}^2\Gamma_{22} + 2L_{G}^2(L_{H}+1)^2\beta_{k}^2\|\yhat_{k}\|^2\notag\\ 
&\quad + \frac{L_{G}\beta_{k}}{\mu_{G}}\|\xhat_{k}\|^2 +  L_{G}^2(L_{H}+2)\beta_{k}^2\|\xhat_{k}\|^2,
\end{align*}
which concludes our proof. 
\subsection{Proof of Lemma \ref{lem:x+y_bound}}
Let $\zhat_{k} = [\xhat_{k}^T, \yhat_{k}^T]^T$. Adding \eqref{lem:xhat:Ineq} to \eqref{lem:yhat:Ineq} yields
\begin{align}
\Eset\left[\|z_{k+1}\|^2\,|\,\Qcal_{k}\right] 
& \leq \left(1-\mu_{F}\alpha_{k} \right) \|\xhat_{k}\|^2 +  \beta_{k}^2\Gamma_{22} +  \alpha_{k}^2\Gamma_{11} + 4L_{H}^2\gamma_{k}\Gamma_{22}\notag\\ 
&\quad + L_{H}^2\left(2L_{G}^2\beta_{k}^2 + \alpha_{k}^2+ 4L_{G}^2\gamma_{k}\right)\|\xhat_{k}\|^2 \notag\\
&\quad + 2L_{H}^2L_{G}^2(L_{H}+1)^2\left(\beta_{k}^2 + 2\gamma_{k} \right)\|\yhat_{k}\|^2 + \left(1-\mu_{G}\beta_{k}\right)\|\yhat_{k}\|^2 + \beta_{k}^2\Gamma_{22}\notag\\ 
&\quad + 2L_{G}^2(L_{H}+1)^2\beta_{k}^2\|\yhat_{k}\|^2  +  \frac{L_{G}\beta_{k}}{\mu_{G}}\|\xhat_{k}\|^2 +  L_{G}^2(L_{H}+2)\beta_{k}^2\|\xhat_{k}\|^2\allowdisplaybreaks\notag\\
&\leq \|\zhat_{k}\|^2 -\mu_{G}\beta_{k}\|\yhat_{k}\|^2 - \mu_{F}\alpha_{k}\|\xhat_{k}\|^2 +  \frac{L_{G}\beta_{k}}{\mu_{G}}\|\xhat_{k}\|^2 \notag\\
&\quad +  2\beta_{k}^2\Gamma_{22} +  \alpha_{k}^2\Gamma_{11} + 4L_{G}^2\gamma_{k}\Gamma_{22}\notag\\
&\quad + 2L_{H}^2L_{G}^2(L_{H}+1)^{2}\left(\big(\frac{1}{L_{H}^2}+1\big)\beta_{k}^2 +   2\gamma_{k} \right)\|\yhat_{k}\|^2\notag\\
&\quad + L_{G}^2(L_{H}+1)^2\left(2\beta_{k}^2 + \frac{\alpha_{k}^2}{L_{G}^2} + 4\gamma_{k}\right)\|\xhat_{k}\|^2\notag\\
&\leq \|\zhat_{k}\|^2 -\mu_{G}\beta_{k}\|\yhat_{k}\|^2 - \frac{\mu_{F}}{2}\alpha_{k}\|\xhat_{k}\|^2  +  2\beta_{k}^2\Gamma_{22} +  \alpha_{k}^2\Gamma_{11} + 4L_{G}^2\gamma_{k}\Gamma_{22}\notag\\
&\quad + 2L_{G}^2(L_{H}+1)^2\left( \frac{\alpha_{k}^2}{2L_{G}^2}+ 2(L_{H}^2+1)(\gamma_{k}+\beta_{k})\right)\|\zhat_{k}\|^2.\label{lem:x+y_bound:Eq1}
\end{align}
where in the last inequality we use 
\begin{align*}
\frac{\beta_{k}}{\alpha_{k}} \leq \frac{2\mu_{F}\mu_{G}}{L_{G}}\cdot
\end{align*}
For convenience, we denote by $c_{k}$
\begin{align*}
c_{k} = 2L_{G}^2(L_{H}+1)^2\left( \frac{\alpha_{k}^2}{2L_{G}^2}+ 2(L_{H}^2+1)(\gamma_{k}+\beta_{k})\right)\|\zhat_{k}\|^2,
\end{align*}
and let $C_{2}$ be defined in \eqref{lem:x+y_bound:C}. Using the relation ${1+x\leq \exp(x)}$ for any $x\geq 0$ and since $\alpha_{k}$ and $\beta_{k}$ satisfy \eqref{lem:x+y_bound:stepsize} we have
\begin{align}
\prod_{t=k}^{\infty}\left(1+c_{t}\right) &\leq \exp\left(\sum_{t=0}^{\infty}c_{t}\right) \leq  \exp\left(C_{2}\right) < \infty.\label{lem:x+y_bound:Eq2}
\end{align}
Let $w_{k}$ be defined as 
\begin{align*}
w_{k} = \prod_{t = k}^{\infty}\left(1 + c_{t}\right)\Eset\left[\|\zhat_{k}\|^2\right]
\end{align*}
Taking the expectation both sides of \eqref{lem:x+y_bound:Eq1} and then multiplying by a finite number 
\[\prod_{t=k+1}^{\infty}\left(1 + c_{t}\right),\]
we obtain 
\begin{align*}
w_{k+1}
&\leq w_{k} +   \left(\beta_{k}^2\Gamma_{22} +  \alpha_{k}^2\Gamma_{11} + \beta_{k}^2\Gamma_{22} + 2\gamma_{k}\Gamma_{22}\right)\prod_{t=k+1}^{\infty}\left(1 + c_{t}\right)\notag\\
&\leq w_{k} +   \exp\left(C_{2}\right)\left(2\beta_{k}^2\Gamma_{22} +  \alpha_{k}^2\Gamma_{11} + 2\gamma_{k}\Gamma_{22}\right)
,    
\end{align*}
where we dropped the negative term on the right-hand side of \eqref{lem:x+y_bound:Eq1}. Summing up both sides of the preceding equation over $k = 0,\ldots,K$ for some $K>0$ yields
\begin{align}
w_{K+1}
&\leq w_{0} + \exp\left(C_{2}\right)\Big(2\Gamma_{22}\sum_{k=0}^{K}\beta_{k}^2 +  \Gamma_{11}\sum_{k=0}^{K}\alpha_{k}^2 + \Gamma_{22}\sum_{k=0}^{K}\gamma_{k}\Big)\notag\\
&\leq w_{0} + C_{1}\exp\left(C_{2}\right)\left(3\Gamma_{22} + \Gamma_{11}\right)\notag\\
&\leq C_{2}\Eset\left[\|\zhat_{0}\|^2\right] + C_{1}\exp\left(C_{2}\right)\left(3\Gamma_{22} + \Gamma_{11}\right). \label{lem:x+y_bound:Eq2b}    
\end{align}
Diving both sides of \eqref{lem:x+y_bound:Eq2b} by $\prod_{t=K+1}^{\infty}\left(1 +c_{t}\right)$ gives \eqref{lem:x+y_bound:Ineq}, i.e., 
\begin{align*}
\Eset\left[\|\zhat_{k+1}\|^2\right]
&\leq \frac{C_{2}\Eset\left[\|\zhat_{0}\|^2\right]}{\prod_{t=K+1}^{\infty}\left(1 +c_{t}\right)}  + \frac{ C_{1}\exp\left(C_{2}\right)\left(3\Gamma_{22} + \Gamma_{11}\right)}{\prod_{t=K+1}^{\infty}\left(1 +c_{t}\right)}\notag\\
&\leq \frac{ C_{2}\Eset\left[\|\zhat_{0}\|^2\right] }{\exp(-C_{2})} + \frac{ C_{1}\exp\left(C_{2}\right)\left(3\Gamma_{22} + \Gamma_{11}\right)}{\exp(-C_{2})}, 
\end{align*}
where the second inequality is due to the relation $1+x\geq \exp(-x)$, $\forall x\geq 0$ and the lower bound of the integral test
\begin{align*}
\prod_{t=K+1}^{\infty}\!\!\!\left(1 +c_{t}\right) \geq \exp\big\{-\sum_{t=K+1}^{\infty}c_{t}\big\}\geq \exp\left(-C_{2}\right).
\end{align*}

\section{Concluding Remarks}
In this paper, we studied nonlinear two-time-scale stochastic approximation, where our main contribution is to provide an explicit formula to characterize the finite-time performance of this method. We showed that the mean square error generated by this method converges at a rate $\Ocal(1/k^{2/3})$. Our analysis is mainly motivated by the classic control theory for singularly perturbed systems, where we utilize a Lyapunov function to characterize the the coupling between fast and slow variables. As mentioned, two-time-scale SA has broad applications in many areas, especially in reinforcement learning. A natural extension from our work is to study the performance of nonlinear two-time-scale SA under Markovian noise, which is often the case in reinforcement learning. Another interesting question is whether one can improve the rate $\Ocal(1/k^{2/3})$ to $\Ocal(1/k)$ as achieved in the case of linear two-time-scale SA. Such a question, however, is not trivial since the existing technique in the linear case may not be applicable to the nonlinear counterpart \cite{KondaT2004,Kaledin_two_time_SA2020}.

\bibliographystyle{IEEEtran}
\bibliography{refs}




\end{document}